\documentclass[12pt]{amsart}
\usepackage{color} 
\usepackage[all]{xy}
\usepackage{amscd,amsmath,latexsym,amssymb}
\usepackage[mathscr]{euscript}
\usepackage{enumitem}

\date{January 23, 2023}

\calclayout
\newtheorem{dummy}{anything}[section]
\newtheorem{theorem}[dummy]{Theorem}
\newtheorem*{thma}{Theorem A}
\newtheorem*{thmaa}{Theorem A$'$}
\newtheorem*{thmb}{Theorem B}
\newtheorem*{thmc}{Theorem C}
\newtheorem{lemma}[dummy]{Lemma}
\newtheorem{proposition}[dummy]{Proposition}
\newtheorem{corollary}[dummy]{Corollary}

\theoremstyle{definition}
\newtheorem{definition}[dummy]{Definition}
  \newtheorem{example}[dummy]{Example}
  
  \newtheorem{remark}[dummy]{Remark}
   
    \newtheorem*{question}{Question}
  \newtheorem*{acknowledgement}{Acknowledgement}

\newcommand
{\eqncount}{\setcounter{equation}{\value{dummy}}%
\addtocounter{dummy}{1}}

\newcommand{\cB}{\mathcal B}

\newcommand{\cF}{\mathcal F}

\newcommand{\bZ}{\mathbb Z}

\newcommand{\bF}{\mathbb F}

\newcommand{\bQ}{\mathbb Q}
\newcommand{\bbR}{\mathbb R}

\newcommand{\bK}{\mathbb K}
\newcommand{\wX}{\widetilde X}
\newcommand{\wK}{\widetilde K}
\newcommand{\scF}{\mathscr F}

\newcommand{\wH}{\widehat H}
\DeclareMathOperator{\Hom}{Hom}

\DeclareMathOperator{\Image}{im}
 \DeclareMathOperator{\Ext}{Ext}

\DeclareMathOperator{\Def}{Def}

 \DeclareMathOperator{\rk}{rank}

\newcommand{\cy}[1]{\bZ/{#1}\bZ}
\newcommand{\la}{\langle}
\newcommand{\ra}{\rangle}

\newcommand{\bd}{\partial}
\newcommand{\id}{\mathrm{id}}

\newcommand{\La}{\Lambda}
\newcommand{\mmatrix}[4]{\bigg (\hskip-4pt\vcenter
{\xymatrix@C-2pc@R-2pc{#1&#2\\#3&#4} }\hskip-2pt
\bigg )}
\newcommand{\fake}{{\textup{D2}-complex}}

\newcommand{\ZG}{\bZ G}
\DeclareMathOperator{\Met}{Met}

\newcommand{\wN}{\widehat N}

\DeclareMathOperator{\adj}{ad}
\DeclareMathOperator{\Sharp}{\sharp}

\begin{document}

\title[Minimal Euler Characteristics]{Minimal Euler Characteristics for Even-Dimensional Manifolds with Finite Fundamental Group}

\author{Alejandro Adem}
\address{\vbox{\hbox{Department of Mathematics, The University of British Columbia} \hbox{Vancouver, British Columbia, Canada}}}
\email{adem@math.ubc.ca}

\author{Ian Hambleton}
\address{\vbox{\hbox{Department of Mathematics \& Statistics, McMaster University}\hbox{Hamilton, Ontario L8S 4K1, Canada}}}

\email{hambleton@mcmaster.ca }

\thanks{The second author was supported by an NSERC Discovery Grant.}

\begin{abstract} 
We consider the Euler characteristics $\chi(M)$  of closed orientable topological $2n$--manifolds with $(n-1)$--connected universal cover and  a given fundamental group $G$ of type $F_n$. We define
$q_{2n}(G)$, a generalized version of the Hausmann-Weinberger invariant \cite{Hausmann:1985} for 4--manifolds, as the minimal value of $(-1)^n\chi (M)$.
For all $n\geq 2$, we establish a strengthened and extended 
version of their estimates, in terms of
explicit cohomological invariants of $G$. 
As an application we obtain new restrictions for non-abelian finite groups arising as fundamental
groups of rational homology 4--spheres. 
\end{abstract}

\maketitle

\tableofcontents
\section{Introduction}\label{sec:one}
In this paper we address the following problem: if $M$ denotes a closed, orientable even-dimensional manifold with a given fundamental
group $G$, then what restriction does this impose on the Euler characteristic of $M$~? In the particular case when
$\chi (M)=2$ we have the related problem of determining which finite groups can be the fundamental group of a closed topological
$2n$-manifold $M$ with the rational homology of the $2n$-sphere (see previous work on the $4$-dimensional case by Hambleton-Kreck \cite{hk2} and Teichner \cite{Teichner:1988}). 

We introduce the following invariant  for discrete groups, extending a definition due to Hausmann and Weinberger \cite{Hausmann:1985} for 4--manifolds: 
\begin{definition}
Given a finitely presented  group $G$, define $q_{2n}(G)$ as the minimum value of $(-1)^n\chi (M)$
for a closed orientable
$2n$--manifold $M$ with $(n-1)$--connected universal cover, such that $\pi_1(M)=G$. 
\end{definition}

We will first assume that $G$ is a finite group. Recall  that Swan  \cite[p.~193]{Swan:1965} defined an invariant $\mu_k(G)$, for each $k \geq 1$,  by the condition that
 $(-1)^k|G|\mu_k (G)$ 
 is the minimal value over all partial Euler characteristics of
a free  resolution of $\bZ$ truncated after degree $k$. We call this a $k$-step resolution.  However, since projective $\ZG$-modules are locally free \cite[\S 8]{Swan:1960a}, $k$-step projective resolutions can be used instead to define $\mu'_k(G) \leq \mu_k(G)$  (see  \cite[Remark, p.~195]{Swan:1965}).

 Let $e_n(G)$ denote the least integer greater than or equal to all the numbers 
$$\dim H^n(G,\bF) - 2 \left ( \dim H^{n-1}(G,\bF) - \dim H^{n-2}(G,\bF) + \dots + (-1)^{n-1}\dim H^0(G,\bF)\right )$$
where the coefficients range over $\bF = \bQ$ or   $\bF = \bF_p$ for all primes $p$.
Our main result is the following:

\begin{thma}\label{thma} If $G$ is a finite group and $n \geq 2$,   then
$$\max \{e_n(G), \mu'_{n}(G)-\mu'_{n-1}(G)\}\le q_{2n}(G)\le 2 \mu_{n}(G).$$
\end{thma}

\begin{remark}\label{rem:except}  By \cite[Theorem 5.1]{Swan:1965}, $\mu'_k(G) = \mu_k(G)$  unless $G$ has periodic cohomology of (necessarily even) period dividing $k+1$,  and $G$ admits no periodic free resolution of period $k+1$.  In this case $k\geq 3$ is odd, and we will  say that the pair $(G,k)$ is \emph{exceptional} (see  Remark \ref{rem:excepttwo}).  For example, 
$\mu'_3(G) < \mu_3(G)$ for some of the $4$-periodic groups $G = Q(8p,q)$ in Milnor's list (see  the calculations in \cite{Madsen:1983,Milgram:1985}). If  $(G,n)$ is an exceptional pair, we provide information about  $q_{2n}(G)$ in Theorem B and Remark \ref{rem:exceptthree} below.  
\end{remark}

The invariants $e_n(G)$ and the  $\mu_k(G)$, for $1\leq k \leq n$,  can also be defined for infinite discrete groups of \emph{type $F_n$}, meaning that there is a model for $K(G,1)$ with finite $n$-skeleton. 
In this case, we obtain similar estimates with a slightly weaker lower bound. Recall that a finitely presented group $G$ is said to be \emph{good} if topological surgery  with fundamental group $G$ holds in dimension four (see Freedman-Quinn \cite[p.~99]{freedman-quinn1}).

\begin{thmaa}\label{thmas}
If $G$ is an infinite discrete group of type $F_n$ with $n\geq 2$,   then
$$\max \{e_n(G), \mu_{n}(G)-\mu''_{n-1}(G)\}\le q_{2n}(G).$$
If  $n \geq 3$, or $n=2$ and $G$ is good, then $q_{2n}(G)\le 2 \mu_{n}(G)$.
\end{thmaa}

The invariants  $\mu''_k(G) = \mu_k(G)$, for $k \geq 3$, and we define $\mu''_2(G) = 1 - \Def(G)$, and $\mu''_1(G) = d(G) -1$, where $\Def(G)$ is the \emph{deficiency} of $G$, defined as the maximum difference $d-r$ of  numbers of generators minus relations over all finite presentations of $G$ (see \cite{Epstein:1961}) and $d(G)$ denotes the minimal number of generators for $G$.
 These modifications to the previous invariants arise from the additional condition that the resolutions be \emph{geometrically realizable} 
(see Section \ref{sec:infinite}). For $k=2$, determining the relation between $\mu_2(G)$ and $1-\Def(G)$ is part of Wall's (unsolved) D2 problem \cite[Section 2]{Wall:1965}, which for infinite groups is related to the Eilenberg-Ganea conjecture \cite{Eilenberg:1957}.

\medskip
Our results sharpen and generalize the estimate proved by Hausmann-Weinberger  \cite[Th\'eor\`eme 1]
{Hausmann:1985}:
$$e_2(G) \le q_4(G) \leq 2(1- \Def(G)), $$
since $\mu_2(G) \leq (1- \Def(G))$ by \cite[Proposition 1]{Swan:1965}. The results of Kirk-Livingston \cite{Kirk:2005} for $q_4(\bZ^n)$ show that these bounds can  be improved for specific groups.

The proof of the lower bound in Theorem A for $q_{2n}(G)$  is given in Section \ref{sec:two}. 
In Section \ref{sec:three} we  establish the upper bound $q_{2n}(G) \leq 2\mu_n(G)$, by generalizing the well-known ``thickening" construction for groups $G$  which admit a \emph{balanced presentation} with equal numbers of generators and relations (i.e.~$\Def(G) =0$).  For $n=2$ this involves showing that  \emph{finite \fake es} with  \emph{good} fundamental groups (e.g.~groups of finite order)  admit suitable thickenings via methods from topological surgery 
(see Theorem \ref{prop:fakeembed}).

\begin{example}
For $E_k=(\bZ/p\bZ)^k$ an elementary abelian $p$-group, $e_n(E_k)=\mu_n(E_k)-\mu_{n-1}(E_k)$, and this number
can be explicitly computed using the Kunneth formula (see Example \ref{ex:threetwelve}). This can be used to show that
$q_{2n}(E_k)$ grows like a polynomial of degree $n$ in $k$, for example
$$\frac{k^4-2k^3+11k^2-34k+48}{24}
\le q_8(E_k)\le \frac{k^4+2k^3+11k^2-14k+24}{12}.$$ 
\end{example}

For $n$ even, $q_{2n}(G)\ge 2$, as the minimal possible Euler characteristic that can occur in our setting is $\chi(M) = 2$, which holds when $M$ has the  rational homology of a $2n$-sphere, and is implied by Theorem A if $\mu_n(G) =1$. 
The condition $\mu_2(G) = 1$ also holds for groups of deficiency zero and there are many groups with this property 
(see \cite{Wamsley:1970}).  In contrast, our computations for the  groups $E_k=(\bZ/p\bZ)^k$ show that 
$q_{4n}(E_k)>2$ for all $n>1$ and $k\ge3$. Hence higher dimensional rational homology spheres with elementary abelian fundamental group of rank larger than 2 cannot occur. 

For periodic groups we can compute $q_{2n}(G)$ in certain cases, which in particular provides an alternate argument for \cite[Corollary 4.4]{hk2} and generalizes that result to higher dimensions: 
\begin{thmb} Let $G$ be a finite periodic group of (even) period $q$. Then $q_{2n}(G) = 2$ if $q$ divides $n+2$, and $q_{2n}(G) = 0$ if $2q$ divides $n+1$. 
\end{thmb}

\begin{remark}
Note that in our setting, $\chi(M)>0$ if and only if $n$ is even (see Corollary \ref{cor:euler}).  
Thus for $n$ odd, the minimal possible value of $q_{2n}(G)= - \chi(M)$ is zero.  Apart from the results of Theorem B for periodic groups with twice their period dividing $n+1$, any finite group $G$ which acts freely and homologically trivially on some product $S^n \times S^n$ will have $q_{2n}(G) = 0$. 
There are many such examples, including any products $G = G_1 \times G_2$  of periodic groups, many rank two finite $p$-groups, including the extra-special $p$-groups of order $p^3$, and all the finite odd order subgroups of the exceptional Lie group $G_2$ (see \cite{Hambleton:2006,Hambleton:2009,Hambleton:2010}). 
\end{remark}

\medskip
We are especially interested in the case of rational homology $4$-spheres (called $\bQ S^4$ manifolds) with finite fundamental group. 
In Section \ref{sec:four} we consider the following ``inverse" problem, for which the lower bound implies significant restrictions on $G$.

\begin{question} Which finite groups can be the fundamental group of a closed topological
$4$-manifold $M$ with the rational homology of the $4$-sphere ?
\end{question}
 For example, it was observed in \cite[p.~100]{hk2} that  if  $G$ is finite abelian, then $d(G) \leq 3$ (see Corollary \ref{cor:twoseven}).  This bound follows directly by estimating the Hausmann-Weinberger invariant $q_4(G)$. Moreover, Teichner \cite[4.13]{Teichner:1988} showed that this bound is best possible for abelian groups by explicit construction of examples.

\smallskip
Our methods  shed light on more complicated finite groups by making use of cohomology with twisted coefficients to obtain better
lower
bounds for $q_4(G)$: 

\begin{thmc}
Let $U_k=E_k\times_T C$ where $p$ is an odd prime, $E_k=(\bZ/pZ)^k$ and $C$ cyclic of order prime to $p$ acts on each $\bZ/p\bZ$ factor in $E_k$ via $x\mapsto x^q$ where 
$q$ is a unit in $\bZ/p\bZ$. 
\begin{enumerate} 
\item If $x^{q^2}\ne x$ for all $1\ne x\in E_k$, then for all $k>4$, $U_k$ does not arise as the
fundamental group of any rational homology 4--sphere.
\item If $q=p-1$, then for all $k>1$, $U_k$ does not arise as the
fundamental group of any rational homology 4--sphere. 
\end{enumerate}
\end{thmc}

This paper is organized as follows: in Section \ref{sec:two} we analyze free group actions on $(n-1)$--connected
$2n$--manifolds using cohomological methods; in Section \ref{sec:three} we discuss minimal complexes and thickenings; in Section \ref{sec:infinite} we prove Theorem A$'$;  in Section \ref{sec:four} we focus on rational homology 4--spheres; and in Section \ref{sec:five} we collect some remarks, examples and questions related to the invariants introduced here. Appendix A contains the proof of Theorem \ref{prop:fakeembed}.

\begin{acknowledgement}  The authors would like to thank Mike Newman and \"Ozg\"un \"Unl\"u for showing that some of the finite groups $G$ with $\mu_2(G) =1$, considered in Example \ref{ex:fivenine},  do have  deficiency zero. We would also like to thank the referee for many valuable comments.
\end{acknowledgement}

\section{Free actions on $(n-1)$--connected $2n$--manifolds}\label{sec:two}
In this section we will apply the cohomological approach outlined 
in \cite[\S 2]{Adem:2019}.
The proofs of \ref{prop:stable}, \ref{prop:extclass} and \ref{prop:exponent} are straightforward
modifications of the results there and details are omitted.
We assume that $Y$ is a closed, orientable, $(n-1)$--connected 
$2n$--manifold with the free orientation-preserving action of a finite group $G$; its homology has a corresponding 
$\bZ G$--module structure. Both $H_{2n}(Y,\bZ)$ and $H_0(Y,\bZ)$ are copies of the trivial module $\bZ$ whereas $H_n(Y,\bZ)$ is a free abelian group with a $\bZ G$--module structure which, by Poincar\'e duality, must be self--dual as a $\bZ G$--module, i.e. $H_n(Y,\bZ)\cong H_n(Y,\bZ)^*$. 

We assume here that $Y$ admits a finite $G$--CW complex structure, with cellular chain complex
denoted by $C_*(Y)$ (if the action is smooth this is always true, and holds up to $G$-homotopy equivalence in the topological case).

We denote by $\Omega^r(\bZ )$ the $ \bZ  G$-module uniquely defined in the stable category (where $\ZG$-modules are identified up to
stabilization  by projectives) as the $r$--fold dimension--shift of the trivial module $\bZ$.
 We refer to \cite{Adem:2004}
and \cite{Brown:1982}
for background on group cohomology. 

\begin{proposition}\label{prop:stable}
Let $Y$ be an $(n-1)$--connected $2n$-manifold with a free action of a finite group $G$ which
preserves orientation. Then 
there is a short exact sequence in the
stable category of $\bZ G$--modules of the form 
\eqncount
\begin{equation}\label{eq:extclass}
0\to \Omega^{n+1}(\bZ)\to H_n(Y;\bZ)\to \Omega^{-n-1}(\bZ)\to 0 
\end{equation}
\end{proposition}

\begin{corollary}\label{long-exact-sequence}
The short exact sequence \eqref{eq:extclass} yields a long exact sequence in Tate cohomology
\[
\dots\to\wH^{i+n}(G,\bZ)\xrightarrow{\cup\,\sigma} \wH^{i-n-1}(G,\bZ)\to \wH^i(G,H_n(Y;\bZ))
\to \wH^{i+n+1}(G,\bZ)
\to\dots
\]
determined by the class $\sigma\in \wH^{-2n-1}(G,\bZ)$ which is
the image of the generator  $ 1\in \wH^{0}(G,\bZ)\cong\bZ/|G|$.
\end{corollary}

We can analyze this sequence just as was done in \cite[\S 2]{Adem:2019}.
\begin{proposition}\label{prop:extclass}
The cohomology class $\sigma\in \wH^{-2n-1}(G,\bZ) \cong H_{2n}(G,\bZ)$ can be identified with the image of the fundamental class $c_*[Y/G]$ under
the homomorphism $$c_*\colon H_{2n}(Y/G, \bZ)\to H_{2n}(BG,\bZ)$$ induced by the classifying map $c\colon Y/G\to BG$. Under this identification, the class $\sigma$ determines the extension \textup{\eqref{eq:extclass}}.
\end{proposition}
\begin{remark} This property of the extension class was proved for $n=2$ in \cite[Corollary 2.4]{hk2}, and the proof in the general case is similar.
\end{remark}
Similarly the map $\Omega^{n+1}(\bZ)\to H_n(Y,\bZ)$ defines an extension class
$$\epsilon_Y\in H^{n+1}(G, H_n(Y,\bZ))$$
which appears in the long exact sequence above as the image of the generator under the map
$\wH^0(G,\bZ)\to \wH^{n+1}(G,\bZ)$. Algebraically this corresponds to mapping the canonical
defining extension for $\Omega^{n+1}(\bZ)$ (identified with the extension class for the module
of cycles in $C_n(Y)$) to the extension obtained by reducing by the module of boundaries $B_n$:

$$\xymatrix@R-3pt@C-2pt{& B_n \ar@{=}[r]\ar[d] &B_n \ar[d] &&&&\cr
0 \ar[r] & Z_n\ar[d]  \ar[r] & C_n\ar[d]\ar[r] & 
C_{n-1} \ar@{=}[d]\ar[r] & \dots \ar[r] & C_0 \ar@{=}[d]\ar[r]&\bZ\ar@{=}[d]\ar[r]&0 \cr
0 \ar[r] & H_n(Y,\bZ) \ar[r] & C_n/B_n\ar[r] & C_{n-1} \ar[r] & \dots \ar[r] & C_0\ar[r]&\bZ\ar[r]&0 }$$
These two extension classes are related as follows:
\begin{proposition}\label{prop:exponent}
Let $G$ denote a finite group acting freely on an $(n-1)$--connected orientable $2n$-manifold $Y$ preserving orientation, then $\epsilon_Y\ne 0$ and $|G| = \rm{exp} (\sigma)\cdot \rm{exp} (\epsilon_Y)$. The class $\epsilon_Y$ has exponent $|G|$ if and only if $\sigma =0$, in which case we have a stable
equivalence 
$$H_n(Y,\bZ) \cong \Omega^{n+1}(\bZ)\oplus \Omega^{-n-1}(\bZ).$$
\end{proposition}
\begin{example}
Observing that the cohomology of a group with periodic cohomology is always zero in odd dimensions,
we see that if $G$ has periodic cohomology, then there is a stable equivalence
$H_n(Y,\bZ)\cong \Omega^{n+1}(\bZ)\oplus \Omega^{-n-1}(\bZ)$.
\end{example}

We note the standard identity $\chi (Y) = 2 + (-1)^n\dim H_n(Y,\bQ)$, and the formula $|G|\chi (Y/G) = \chi (Y)$ from the covering $Y \to Y/G$.   Since the transfer map induces an isomorphism $H_i(Y/G;\bQ) \cong H_i(Y;\bQ)^G$, we have
$\chi (Y/G) = 2 + (-1)^n\dim  
H_n(Y, \bQ)^G$.
In particular 
$$\dim H_n(Y,\bQ)^G =(-1)^n\left ( \chi (Y/G) - 2\right).$$
From the stable sequence
$$0\to \Omega^{-n}(\bZ)\to \Omega^{n+1}(\bZ)\to H_n(Y, \bZ)\to 0$$
we infer the existence of projective modules $Q_r$ and $Q_s$
which fit into an exact sequence
\eqncount
\begin{equation}\label{H2}
0\to \Omega^{-n}(\bZ)\oplus Q_s \to \Omega^{n+1}(\bZ)\oplus Q_r
\to H_n(Y;\bZ)\to 0
\end{equation}
where $Q_i \otimes \bQ \cong [\bQ G]^i$ for $i=r,s$. Here we write $\Omega ^{j+1}(\bZ)$ ($j\geq 0$) for the j-th kernel in a minimal projective resolution of $\bZ$, meaning a resolution: 
\eqncount
\begin{equation}\label{H3}
0\to \Omega^{j+1}(\bZ)\to P_j \to  P_{j-1}\to\dots\to P_0\to\bZ\to 0
\end{equation}
realizing $\mu'_j(G)$ (see \cite[p.~193]{Swan:1965}), from which we see that
$$\rk _{\bZ} \Omega^j(\bZ) + (-1)^{j-1} = |G|(\mu'_{j-1}(G))  ~~~\rm{and}~~~ 
\rk _{\bZ} \Omega^j(\bZ)^G +(-1)^{j-1} = \mu'_{j-1}(G) $$
where $(-1)^k|G|\mu'_k (G)$  is precisely the minimal value over all partial Euler characteristics of
a \emph{projective} resolution of $\bZ$ over $\ZG$ (see  \cite[Remark, p.~195]{Swan:1965}). The corresponding invariants $\mu_k(G)$ for minimal \emph{free} resolutions of $\bZ$ were
defined by Swan (see \cite[p.~193]{Swan:1965}).

By dualizing we see
that a minimal representative for $\Omega^{-j}(\bZ)$ is given by
$\Omega^j(\bZ)^*$, the dual module.  Thus for our purposes we have
$$\rk _{\bZ} \Omega^{n+1}(\bZ)^G = \mu'_{n}(G)+ (-1)^{n+1}, ~~\rk _{\bZ} \Omega^{-n}(\bZ)^G = 
\mu'_{n-1}(G)+(-1)^n.$$

Applying invariants after tensoring over $\bQ$ to the exact sequence \eqref{H2} yields the formula
$$\mu'_{n}(G)+ (-1)^{n+1} + r = \mu'_{n-1}(G)+(-1)^n + s + (-1)^n[ \chi (Y/G) - 2].$$
whence we obtain
$$s -r  = \mu'_n(G) - \mu'_{n-1}(G) + (-1)^{n+1} \chi (Y/G)$$

\begin{theorem}\label{thm:invariant}
If $Y$ is a closed, $(n-1)$--connected $2n$--manifold with a free, orientation-preserving action of $G$, a finite group, then
for any subgroup $H\subset G$
$$\mu'_n(H) - \mu'_{n-1}(H) \le (-1)^n [G:H] \chi (Y/G).$$
\end{theorem}
\begin{proof}
We will prove this for $H=G$ by contradiction.
Assume that $s-r >0$ and form the diagram
$$\xymatrix@R-3pt@C-3pt{& Q_s \ar@{=}[r]\ar[d] &Q_s \ar[d] &&\cr
0 \ar[r] & \Omega^{-n}(\bZ)\oplus Q_s\ar[d]  \ar[r] & \Omega^{n+1}(\bZ) \oplus Q_r \ar[d]\ar[r] & 
H_n(Y,\bZ) \ar@{=}[d]\ar[r] & 0 \cr
0 \ar[r] & \Omega^{-n}(\bZ) \ar[r] & L \ar[r] & H_n(Y, \bZ) \ar[r] & 0 
}$$
where $L$ is the quotient of $\Omega^{n+1}(\bZ) \oplus Q_r$ in the middle vertical exact sequence. Note that this middle vertical exact seqence splits (since $L$ is torsion-free). 
Hence $$\Omega^{n+1}(\bZ) \oplus Q_r \cong L \oplus Q_s.$$
By Swan \cite[Lemma 2.1]{Swan:1960b}, there is a projective resolution
$$ 0 \to L \to P_{n}\oplus Q_r \to P_{n-1} \oplus Q_s  \to P_{n-2}  \to \dots \to P_0\to\bZ\to 0$$
Since $s>r$, this contradicts the minimality of the resolution \eqref{H3} realizing $\mu'_n(G)$.
Hence 
we have shown that $s-r\le 0$. The full result follows using covering spaces.
\end{proof}

\begin{remark}\label{rem:excepttwo} As mentioned in the Introduction, Swan proved that $\mu'_k(G) = \mu_k(G)$  unless $G$ has periodic cohomology of period  dividing $k+1$,  and $G$ admits no periodic free resolution of period $k+1$. In these \emph{exceptional} cases, $\mu_k (G) = 1$ and $\mu'_k(G) = 0$. In contrast, 
$\mu_k (G) = 0$ if $G$ has a periodic free resolution of period $k + 1$ and $G\ne 1$. 
We  also note that if the pair $(G,k)$ is exceptional, then $k\geq 3$ is odd and $G$ is non-cyclic. In particular, $\mu'_k(G) = \mu_k(G)$ if $G$ is a finite $p$-group (see  \cite[Corollary 5.2]{Swan:1965}).
\end{remark}

If the pair $(G,n)$ is not exceptional, the numbers $\mu_n(G)$ can be computed using group cohomology. By a result of Swan  \cite[Proposition 6.1]{Swan:1965},  
the invariant $\mu_n(G)$ is the least integer greater than or equal to all the numbers
$$(\dim  M)^{-1}\left ( \dim H^n(G,M) - \dim H^{n-1}(G,M) + \dots
+ (-1)^n \dim H^0(G,M)\right )$$
as $M$ ranges over all simple $\bF_pG$--modules for all
primes $p$ dividing $|G|$. As extending the field doesn't change dimensions we can take $\mathbb K_p$ an algebraically closed field
of characteristic $p$ and restrict attention to absolutely irreducible $\mathbb K_pG$--modules. Next we introduce

\begin{definition}\label{def:en}
For any discrete group $G$ of type $F_n$, let $e_n(G)$ denote the least integer greater than or equal to all the numbers 
$$\dim H^n(G,\bF) - 2 \left ( \dim H^{n-1}(G,\bF) - \dim H^{n-2}(G,\bF) + \dots + (-1)^{n-1}\dim H^0(G,\bF)\right )$$
where the coefficients range over $\bF = \bQ$ or   $\bF = \bF_p$ for all primes $p$.
\end{definition}

\begin{remark}When $G=P$ is a finite $p$--group, the trivial module
$\bF_p$ is the only simple module, and we can verify that
$\mu_n (P)-\mu_{n-1}(P) = e_n(P)$. 
\end{remark}

We have the following elementary inequality: 
\begin{lemma}\label{lem:twoone}
Suppose that $X$ is a closed orientable $2n$--manifold with fundamental group $G$ of type $F_n$ whose universal
cover is $(n-1)$--connected. Then for any subgroup $H\subset G$ of finite index
$$e_n(H) \le [G:H] (-1)^n\chi (X).$$
\end{lemma}
\begin{proof}
Let $\bF$ denote any field of coefficients. The connectivity of the universal cover implies that $$H^i(G,\bF)\cong H^i(X,\bF)~~\text{for}~~ 0\le i\le n-1$$
and $$\dim H^n(G,\bF)\le \dim H^n(X,\bF).$$
By Poincar\'e duality we have 
$$H^k(X,\bF) \cong H^{2n-k}(G,\bF)~~\text{for}~~ n+1\le k \le 2n.$$  Combining these facts and using covering
space theory we
obtain the inequality.
\end{proof}
Applying the mod $p$ coefficient sequence yields an attractive corollary
\begin{corollary}
If $X$ is a closed orientable $2n$ manifold with finite fundamental group $G$ and 
$(n-1)$--connected universal cover, then for all primes $p$ dividing $|G|$, and subgroups $H\subset G$,
$$\dim H^{n+1}(H,\bZ)\otimes\bF_p - \dim H^n(H,\bZ)\otimes\bF_p \le (-1)^n([G:H]\chi (X)-2).$$
\end{corollary}

\begin{proof} Since $[G:H]\chi (X)$ equals the Euler characteristic of the  $[G:H]$-fold covering of $X$, it is enough to do the case $H=G$. Let $h^i(G) = \dim  H^i(G;\bF_p)$.  From the relations noted above, and Lemma \ref{lem:twoone},  we have the formula
$$(-1)^n(\chi(X) -2) 
 \geq
 h^n(G) - 2\sum_{i =1}^{n-1} (-1)^{i+1} h^{n-i}(G).$$
But by the mod $p$ coefficient sequence, we have 
$$h^i(G) = \dim H^{i+1}(G,\bZ)\otimes\bF_p + \dim H^i(G,\bZ)\otimes\bF_p, \text{\ for\ } 1 \leq i \leq n.$$
The result follows by combining these two relations. 
\end{proof}
\noindent Applying this to any subgroup $C\subset G$ of prime order, we obtain
\begin{corollary}\label{cor:euler}
If $X$ is a closed orientable $2n$--manifold with $(n-1)$--connected universal cover and
non--trivial finite fundamental group $G$, then $\chi (X) >0$ if and only if $n$ is even.
\end{corollary}
\begin{proof} Let $C\subset G$ be a cyclic subgroup of order $p$, a prime. Then 
$H^{2k}(C;\bZ) = \bZ/p\bZ$ (if $k >0$),  and  $H^{2k+1}(C;\bZ) = 0$. For $n$ even, applying the inequality above with $H = C$ yields $\chi (X)>0$. When $n$ is odd, note that $b_n(X) \neq 0$ implies $b_n(X) \geq 2$, since the intersection form of $X$ is non-singular and skew-symmetric. Hence $\chi (X)\le 0$.
\end{proof}

\section{Minimal $K(G,n)$-complexes and thickenings}\label{sec:three}

We now turn our attention to the \textbf{existence} of orientable $2n$--manifolds having fundamental
group  of type $F_n$ and $(n-1)$--connected universal cover. We recall the following well-known
construction (see Kreck and Schafer \cite[\S 2]{kreck-schafer1}): 

\begin{proposition}\label{prop:existence}
Let $G$ be a discrete group  of type $F_n$ for $n\ge 2$. Then there exists a closed orientable $2n$--manifold  
$Z$ such that $\pi_1(Z) =G$ with $(n-1)$--connected universal cover. 
\end{proposition}
\begin{proof}
Let $K$ denote a finite CW complex of dimension $n$ with $\pi_1(K) = G$ whose universal covering is $(n-1)$-connected. For example, 
take a finite, cellular model for the classifying space $BG$, and consider its $n$-skeleton $K$. Then
we can construct a smooth $2n$-manifold $Z=M(K)$ by doubling
a $2n$-dimensional handlebody thickening of $K$.  Thus the universal cover $\tilde{Z}$ of $M(K)$ is an $(n-1)$--connected, closed orientable $2n$--manifold with a free
action of $G$ such that 
$$\pi_n(M(K)) \cong H_n(M(K); \La) \cong H^n(K; \La) \oplus H_n(K; \La), $$ 
where $\La:= \ZG$ denotes the integral group ring. Moreover, the Euler characteristic $\chi(M(K)) = 2 \chi(K)$.
A variation of this construction is to let $Z$ denote the boundary of a regular neigbourhood. for some embedding  $K \subset \bbR^{2n+1}$  of the finite $n$-complex in Euclidean space. 
\end{proof}

\begin{definition}\label{def:one} Let $G$ be  a discrete group of type $F_n$.
 A finite CW complex $K$ of dimension $n\geq 2$, with fundamental group
$\pi_1(K) = G$ and $\pi_i(K) =0$ for $1\leq i \leq n-1$,  is called a \emph{$K(G,n)$-complex}.
\end{definition}

The chain complex $C_*(\tilde{K})$ of the universal covering of a $K(G,n)$-complex affords a free $n$-step resolution of the trivial $\bZ G$--module $\bZ$. 
Conversely, we wish to realize a given finitely generated $n$-step free resolution
$$ \scF :  F_n \to F_{n-1} \to \dots \to F_1 \to F_0 \to \bZ \to 0$$
as the equivariant chain complex of a suitable $K(G,n)$-complex. Note that by Swan \cite[Theorem 1.2]{Swan:1965}, we have $\mu_n(G) \leq (-1)^n\chi(\scF)$ and that the lower bound is attained by some resolution.

\begin{proposition}\label{prop:threeone}  Let $G$ be  a discrete group of type $F_n$, and let $\scF$ be an $n$-step resolution of $\bZ$  by finitely generated free $ZG$-modules. If $n \geq 3$, then there exists a finite $K(G,n)$-complex $K$ and a $G$-equivariant chain homotopy equivalence $C_*(\wK) \simeq \scF $.
\end{proposition}
\begin{proof} Let $n \geq 3$, we can apply \cite[Lemma 8.12]{Hambleton:2013} to show that $\scF$ is chain homotopy equivalent to a finitely generated free complex $\scF'$ which agrees with the $2$-skeleton of a model for $K(G,1)$. Then the construction of \cite[Lemma 3.1]{Swan:1960b} (credited to Milnor) provides the required complex $K$ by successively attaching $i$-cells equivariantly using the boundary maps from the chain complex  $\scF'$.
\end{proof}

\begin{remark} For finite groups, Swan \cite[Corollary 5.1]{Swan:1965} shows that under certain additional assumptions, one can geometrically realize the actual sequence $f_0, f_1, f_2, \dots$ of ranks for the $i$-chains of $\scF$. We also record the facts due to Swan that $\mu_n(G) \geq 1$ for $n$ even, and $\mu_n(G) \geq 0$ for $n$ odd if $G \neq 1$ is finite (see \cite[\S 1]{Swan:1965}). 
 \end{remark}

\begin{corollary}\label{cor:upperone} If $n \geq 3$, then for any discrete group of type $F_n$ we have $q_{2n}(G) \leq 2\mu_n(G)$. In particular, if $n$ is even and  $G$ is a finite group with $\mu_n(G) =1$, then  $G$ is the fundamental group of a rational homology $2n$-sphere. 
\end{corollary}
\begin{proof} We apply Proposition \ref{prop:threeone} to a minimal $n$-step resolution $\scF$ with $\chi(\scF) =\mu_n(G)$, and obtain a finite $K(G,n)$-complex $K$. The manifold $Z= M(K)$ constructed in Proposition \ref{prop:existence} provides the upper bound $q_{2n}(G) \leq \chi(Z) = 2 \mu_n(G)$.
\end{proof}

 We now consider the case $n=2$, where the argument above fails at the first step. To establish our upper bound for $q_4(G)$ we need a more general construction and some results of C.~T.~C.~Wall \cite{Wall:1965,Wall:1966}.
 
 \begin{definition} A finite complex $X$ \emph{satisfies Wall's \textup{D2}-conditions} if  $H_i(\wX) =0$, for $i >2$, and $H^{3}(X; \cB) = 0$, for all coefficient bundles $\cB$. Here $\wX$ denotes the universal covering of $X$. If these conditions hold, we will say that $X$ is a
\emph{\fake}. If every \fake\ with fundamental group $G$ is homotopy equivalent to a finite $2$-complex, then we say that \emph{$G$ has the \textup{D2}-property}.
 \end{definition}

 In \cite[p.~64]{Wall:1965}, Wall proved that a finite complex $X$ satisfying the D2-conditions is homotopy equivalent to a finite $3$-complex. We will therefore assume that all our D2-complexes have $\dim X \leq 3$. It is not known at present whether all discrete groups have the \textup{D2}-property. 
 Note that $\mu_2(G) \leq (1- \Def(G))$ by \cite[Proposition 1]{Swan:1965}, and equality holds if  $G$ has the \textup{D2}-property.

\begin{proposition}[{\rm \cite{Wall:1965},  \cite[Corollary 2.4]{Hambleton:2019}}] 
\label{prop:threetwo} Any finitely generated free resolution 
$$\scF : F_2  \to F_1 \to F_0 \to \bZ \to 0$$
over $\ZG$ is chain homotopy equivalent to $C_*(X)$, where  $X$ is a  finite \fake.
\end{proposition}
If  we apply this to a minimal resolution with $\chi(\scF) = \mu_2(G) = \mu'_2(G)$, then if $G$ is finite the module $H_2(\wX;\bZ)$ is a minimal  $\bZ$-rank representative of the stable module $\Omega^3(\bZ)$. The following result may also be of independent interest  (it applies to any finitely presented group $G$ which is \emph{good} in the sense of Freedman \cite[p.~99]{freedman-quinn1}, in particular to poly-(finite or cyclic) groups).

\begin{theorem}\label{prop:fakeembed} For any finite \fake\  $X$  with good fundamental group there exists a closed topological $4$-manifold $M(X)$ with $\pi_1(M(X)) = \pi_1(X)$ amd $\chi (M(X)) = 2\chi(X)$.
\end{theorem}

For continuity we defer the proof of this result to Appendix A.

\begin{corollary}\label{cor:upper} For $G$ a  finitely presented good  group, $q_{4}(G) \leq 2\mu_2(G)$. In particular, $\mu_2(G) =1$  and $G$ finite implies that $G$ is the fundamental group of a rational homology $4$-sphere.
\end{corollary}
\begin{proof} We apply Proposition \ref{prop:threetwo} to realize a minimal $2$-step resolution by a finite \fake, and then Theorem \ref{prop:fakeembed} provides a suitable $\bQ S^4$ manifold.
\end{proof}

\begin{proof}[The proof of Theorem A]
Concatenating our previous results, we have obtained the estimates 
$$\max \{e_n(G), \mu'_{n}(G)-\mu'_{n-1}(G)\}\le q_{2n}(G)\le 2 \mu_{n}(G).$$
for any finite group $G$. For the lower bound, we apply Theorem \ref{thm:invariant} and Lemma \ref{lem:twoone}. For the upper bound, 
we apply  Corollary \ref{cor:upperone} if $n >2$, and Corollary \ref{cor:upper} for $n=2$.
\end{proof}

We now prepare for the proof of Theorem B.
The next result, due to Swan and Wall, shows that arbitrary periodic groups appear as fundamental groups of rational homology spheres.
\begin{lemma}\label{lem:periodic}  If $G$ is a finite group with periodic cohomology of period dividing $2k+2$,  then  $\mu_{2k}(G) =1$ for $k \geq 1$.
\end{lemma}
\begin{proof}  We will discuss the case $k=1$ for groups of period $4$.
Swan \cite{Swan:1960b} constructed a finitely dominated Poincar\'e $3$-complex $Y$ with $\pi_1(Y) = G$, and Wall \cite[Corollary 2.3.2]{wall-pc1} shows that $Y$ is obtained from a \fake\ by attaching a single $3$-cell. The chain complex $C_*(\widetilde Y)$ provides a projective resolution
$$\scF' :    P \to F_1 \to F_0  \to \bZ \to 0$$
with $\chi(\scF') =1$, where $P$ is projective, $F_1$ and $F_0$ are free and $I(G)^* =\ker d_2(\scF')$. This shows that $\mu'_2(G) =1$ and so  $\mu_2(G) =  \mu'_2(G) =1$ by Swan's results.

One can give a direct argument for this last step. 
By adding a projective $Q$ so that $P\oplus Q = F$ is free, we obtain a free resolution
$$ \scF : F \to F_1 \to F_0 \to \bZ \to 0$$
with $I(G)^* \oplus Q = \ker d_2(\scF)$. By the ``Roiter Replacement Lemma" (see \cite[Proposition 5]{Roiter:1966}, or
\cite[Theorem 3.6]{Jacobinski:1968}), $I(G)^* \oplus Q = J \oplus F'$, where $F'$ is free and $J$ is locally isomorphic to $I(G)^*$, so $\rk_{\bZ} (J) = \rk_{\bZ}I(G)^*$.
We now  divide out the image of $F'$ in $F$ (a direct summand)  to obtain a free resolution
$$ \scF'' : F'' \to F_1 \to F_0 \to \bZ \to 0$$
with  $J =\ker d_2(\scF'')$ and $\chi(\scF'') =1$. Hence $\mu_2(G) = 1$.
  
 A similar argument shows that $\mu_{2k}(G) =1$, for all $k >1$,   if $G$ has periodic cohomology with period dividing $2k+2$. Details will be left to the reader.
 \end{proof}
\begin{remark}
The calculation in Lemma \ref{lem:periodic} together with Theorem \ref{prop:fakeembed}  provides an alternate proof of \cite[Corollary 4.4]{hk2}.  However, the essential ingredients are the same in both arguments. 
\end{remark}

\begin{proof}[The proof of Theorem B]  By assumption, the group $G$ is periodic of even period $q$. In the first case, if $q$ divides $n+2$, then $n$ is even and  $\mu_n(G) =1$ by Lemma \ref{lem:periodic}. By Theorem A, we have the inequalities 
$$2 \leq q_{2n}(G) \leq 2 \mu_2(G)=2$$ 
and hence $q_{2n}(G) =2$.
 
 In the second case, $n$ is odd and the minimal Euler characteristic $q_{2n}(G)\geq 0$ by Corollary \ref{cor:euler}. We will show that the lower bound is realized when  $G$ is a periodic group of even period $q$, provided that $2q$ divides $n+1$. 
 
 This follows from the solution of the space form problem:
 Madsen, Thomas and Wall  \cite[Theorem 1]{Madsen:1983a},  \cite[Corollary 12.6]{Wall:1979} proved that there exists a finite Poincar\'e duality complex $X$ (called a \emph{finite Swan complex})  of dimension $(2k-1)$, with $\pi_1(X) = G$ and universal covering $\widetilde X \simeq S^{2k-1}$, whenever $k \equiv 0 \pmod{ e(G)}$,  where $e(G)$ is the Artin exponent of $G$ \cite[p.~94]{Lam:1968}.  Moreover, a detailed analysis of the group cohomology of periodic groups shows that $2e(G)$ is equal to $q$ or $2q$, depending on the structure of its $2$-hyperelementary subgroups (see  Wall \cite[p.~542]{Wall:1979}, where the notation $2d(\pi)$ is used for the period of a periodic group $\pi$).
 
 For any finite Swan complex $X$, there exists a degree one normal map $(f, b)\colon N \to X$, where $N^{n}$ is a closed, topological $n$-manifold (see \cite[Proposition 2]{Wall:1978} and \cite[Corollary 3.3]{Thomas:1971}). We then have a degree one normal map of pairs 
$$(f\times \id,  b\times \id)\colon (N \times D^{n+1}, N \times S^n) \to (X \times D^{n+1}, X \times S^n).$$
By Wall's ``$\pi$-$\pi$ Theorem"  \cite[Theorem 3.3]{wall-book}, this normal map is normally cobordant to a homotopy equivalence of pairs. It follows that $X \times S^n$ is homotopy equivalent to a closed topological $2n$-manifold. 
Since $X \times S^n$ has Euler characteristic zero, these examples show that $q_{2n}(G) = 0$ as required. 
\end{proof}
 
 \begin{remark}[Smooth examples] If $G$ satisfies the \emph{$2p$-conditions} (meaning that every subgroup of order $2p$ is cyclic, for $p$ prime), Madsen, Thomas and Wall  \cite[Theorem 5]{Madsen:1983a} proved that there exists a closed, oriented,  smooth $(2k-1)$-manifold $N^{2k-1}$ with $\pi_1(N) = G$ and universal covering $\widetilde N = S^{2k-1}$, whenever $k \equiv 0 \pmod{ e(G)}$. Under this extra assumption, the products $N^n \times S^{n}$, for $n = 2qr-1$ provide \emph{smooth manifolds} realizing the minimium value $q_{2n}(G) = 0$.
 \end{remark} 
%
%

\begin{remark}[The exceptional case]\label{rem:exceptthree} In the arguments above, we have not used the full strength of the Madsen-Thomas-Wall results, which produce smooth space forms in the minimal dimension $q-1$ whenever $q= 2e(G)$ (see the discussion on \cite[p.~142]{Madsen:1983a}). This observation does give additional examples of periodic groups with $q_{2n}(G) = 0$, e.g when $n+1 \equiv 2 \pmod{4}$, but deciding whether $2e(G)$  equals $q$ or $2q$ for a given $G$ involves difficult number theory. 

If the pair $(G. n)$ is exceptional, then surgery theory can be used to study $q_{2n}(G)$ as follows
 (see  \cite[\S\S2-3]{Madsen:1976} for background on the space form problem):
\begin{enumerate}
\item For any periodic group with  period $n+1$, there exists a \emph{finitely dominated} Swan complex $X$ with $\pi_1(X) = G$ and universal covering $\wX \simeq S^n$ (see \cite[Proposition 3.1]{Swan:1960b}). 
\item For any finitely dominated Swan complex $X$, there exists a degree one normal map $(f, b)\colon N \to X$, where $N^{n}$ is a closed, oriented, topological $n$-manifold (see \cite[Proposition 2]{Wall:1978} and \cite[Corollary 3.3]{Thomas:1971}).
\item The product $X \times S^n$ is homotopy equivalent to a finite Poincar\'e complex (by the product formula for Wall's finiteness obstruction \cite[Theorem 0.1]{Gersten:1966}).
\item We have a degree one normal map 
$$(f\times \id,  b\times \id)\colon  N \times S^n \to X \times S^n.$$
 with   surgery obstruction $\lambda(f,b) \in L^h_{2n}(\bZ G)$ determined by the Wall finiteness obstruction 
 $\sigma (X) \in \wK_0(\bZ G)$ (see \cite[p.~244]{Pedersen:1980}).  
 
\item If  $\lambda(f,b) = 0$ (this is the hard step), then  this normal map would be normally cobordant to a homotopy equivalence. In other words,  $X \times S^n$ would be homotopy equivalent to a closed topological $2n$-manifold with Euler characteristic zero.
\end{enumerate}
\end{remark}


 We conclude this section with a sample computation of the estimates for elementary abelian $p$-groups.
\begin{example}\label{ex:threetwelve}  If $E_k=(\bZ/p\bZ)^k$ then we can use the Kunneth formula to compute these invariants.
The term $\mu_n(E_k)$ has a polynomial of degree $n$ as its leading term.
For $n=2, 3, 4$ we have
$$\frac{k^2-3k+4}{2}\le q_4(E_k)\le k^2-k+2 $$
$$\frac{k^3-3k^2+8k-12}{6}\le q_6(E_k)\le \frac{k^3+5k-6}{3}$$
$$\frac{k^4-2k^3+11k^2-34k+48}{24}
\le q_8(E_k)\le \frac{k^4+2k^3+11k^2-14k+24}{12}$$
\medskip\noindent
For instance, for $k=2$ this only gives the rough estimate $1\leq q_8(E_2) \leq 6$, but we know that $q_8(E_2) =2$ by performing surgery\footnote{Here $L^7(\bZ/p\bZ)$ denotes the $7$-dimensional lens space with fundamental group $G= \cy p$, and the surgery is performed on the $S^1$ factor by removing $D^6 \times S^1$ and gluing in $S^5 \times D^2$.} on $L^7(\bZ/p\bZ) \times S^1$. However, for $k=3$ the lower bound gives $q_8(E_3) \geq 3$, and hence $E_3$ is not the fundamental group of a rational homology $8$-sphere. 
\end{example}

\section{The proof of Theorem A$'$
}\label{sec:infinite} 
In this section we establish a lower bound for $q_{2n}(G)$, for $G$ an infinite discrete group of type $F_n$.
With the results of Lemma \ref{lem:twoone}, Corollary \ref{cor:upperone},  and Corollary \ref{cor:upper}, this will complete the proof of Theorem A$'$.

\medskip
The invariants $\mu''_k(G)$ used in the statement  of Theorem A$'$ can also be defined as follows.

\begin{definition}
For $k \geq 2$,
let $\mu''_{k}(G)=(-1)^k \cdot \min\{\chi(\cF)\}$, 
 where $\cF$ varies over all $k$-step resolutions
$$ \cF :  F_k \to F_{k-1} \to \dots \to F_2 \to F_1 \to F_0 \to \bZ\to 0$$
of $\bZ$ by finitely generated free $\ZG$-modules, which arise geometrically as the chain complex of the universal covering for a finite $CW$-complex  of dimension $k$ with fundamental group $G$. 
\end{definition}

The sign $(-1)^k$ is introduced to agree with Swan's conventions. Note the inequalities 
$$\mu'_k(G) \leq \mu_k(G) \leq  \mu''_k(G)$$
relating these invariants to those defined by Swan. We define $\mu''_1(G) = d(G) -1$, where  $d(G)$ denotes the minimal number of generators for $G$.

\begin{remark} By Proposition \ref{prop:threeone}, we have $\mu_k(G) =  \mu''_k(G)$ if $n\geq 3$. Note that $\mu''_2(G) = 1- \Def(G)$.
If  $\mu_2(G) < \mu''_2(G)$ for some finitely presented group $G$, then there would be a counter-example to Wall's D2 problem (but no such examples are known at present). In addition, we do not know if the strict inequality $\mu_1(G) < d(G) -1$ can occur. 
\end{remark}

We now establish the lower bound for infinite groups.
\begin{theorem}\label{thm:infone} Let $G$ be a discrete group of type $F_n$, for $n\geq 2$.
If $Y$ is a closed, $(n-1)$--connected $2n$--manifold with a free, orientation-preserving action of $G$, a finite group, then
for any subgroup $H\subset G$ of finite index
$$\mu_n(H) - \mu''_{n-1}(H) \le (-1)^n [G:H] \chi (Y/G).$$
\end{theorem}

\begin{proof}  It suffices to prove this inequality for $H = G$, and then apply covering space theory. Let $M$ denote a closed, orientable $2n$-dimensional manifold with $n \geq 2$ and  fundamental
group $G$ of type $F_n$, such that $\pi_i(M) = 0 $ for $1 < i <n$. Let $K \simeq M$ be a finite $CW$-complex homotopy equivalent to $M$, and let $C:= C(K; \La) = C(\wK)$  denote the chain complex of its universal covering. It is a finite chain complex, with each $C_i$ a  finitely generated free $\ZG$-module. We note that the homology of $M$ is computed from the chain complex $C \otimes_{\ZG}\bZ$, and therefore $\chi(M) = \sum_{i=0}^{2n} (-1)^i c_i$, where $c_i := \rk_{\ZG} C_i$.

We may assume that the $(n-1)$-skeleton $K^{(n-1)}\subset K$ has $(-1)^{n-1}\chi(\wK^{(n-1)}) = \mu''_{n-1}(G)$, by applying Wall's construction of a \emph{normal form} to replace $K$ by a homotopy equivalent complex if necessary (see \cite[p.~238]{wall-pc1}).

The long exact sequences of the triples $(K, K^{(i)}, K^{(i-1)})$, for cohomology with $\ZG$-coefficients gives:
$$ 0 \to H^{i}(K, K^{(i-1)}) \to H^{i}(K^{(i)},  K^{(i-1)}) \to H^{i +1}(K, K^{(i)}) \to H^{i +1}(K, K^{(i-1)}) \to 0$$
If we let $Z^i := \ker \delta^{i}$ and $B^i := \Image \delta^{i-1}$ (for later use) in the cochain complex $(C^*, \delta^*)$, where $C^i = \Hom_{\La}(C_i, \La)$, then the sequence above becomes
$$0 \to Z^i \to C_i^* \to Z^{i+1} \to H^{i+1}(C) \to 0.$$
Since $H^i(C) = H_{2n-i}(C) = 0$, for $n+1 \leq i \leq 2n-1$, and $H^{2n}(C) = \bZ$, we can splice the short exact sequences
$$ 0 \to  Z^i \to C_i^* \to Z^{i+1} \to 0$$
for $n \leq i \leq 2n-1$, and obtain a long exact sequence
$$ 0 \to Z^n \to C_n^* \to C^*_{n+1} \to \dots \to C^*_{2n-1} \to C^*_{2n} \to \bZ \to 0$$
Since this a resolution of $\bZ$ by finitely generated $\ZG$-modules, with $\rk_{\ZG} (C_i^*) := c_i$, we have 
$$(-1)^n\chi_{upper}(C) := (-1)^n\sum_{i=n}^{2n} (-1)^{i} c_i \geq \mu_{n}(G).$$
On the other hand, by the normal form construction, we have 
$$(-1)^n\chi_{lower}(C) :
= (-1)^n\sum_{i=0}^{n-1} (-1)^i c_i = (-1)^n(-1)^{n-1} \mu''_{n-1}(G) = -\mu''_{n-1}(G).$$
Therefore $q_{2n}(G) \geq (-1) ^n\chi(M) \geq \mu_{n}(G) -\mu''_{n-1}(G)$, as required.
\end{proof}

\begin{remark}
Note that $\mu'_1(G) \leq \mu_1(G) \leq d(G) -1$ by Swan \cite[Proposition 1]{Swan:1965}, so this is slightly different than the estimate in Theorem A for finite groups if $n = 2$. 
\end{remark}

\section{Rational homology 4--spheres}\label{sec:four}
We now specialize our results to the case when $M$ is a rational homology $4$--sphere with finite
fundamental group $G$. We would like to find restrictions on $G$, by computing  $\mu_2(G)-\mu_1(G)$. Note that 
$\mu_1(G)=\mu_1'(G)$ and $\mu_2(G)=\mu_2'(G)$ and that for
any solvable finite group, $\mu_1(G)= d(G)-1$ \cite[Proposition 1]{Cossey:1974}. 							
Let $A$ denote a finite abelian group minimally generated by $d$ elements, then using Theorem \ref{thm:invariant}
we have the estimate
$$\mu_2(A)-\mu_1(A)=\frac{d^2-3d+4}{2}\le \chi (M)=2$$ and so we recover the  estimate proved in \cite[3.4]{Teichner:1988}:

\begin{corollary}\label{cor:twoseven}
If $G$ is a finite abelian group minimally generated by $k>3$ elements, then it cannot be realized 
as the fundamental group of
a closed 4--manifold which is a rational homology sphere.
\end{corollary}

Our next objective will be to consider examples where twisted coefficients can be used to establish conditions
for non--abelian groups. Recall the result due to Swan \cite[Theorem 1.2 and Proposition 6.1]{Swan:1965}: for any finite group $G$, $\mu_n(G)$
is the smallest integer which is an upper bound on
$$(\dim M)^{-1}\left (  h^n(G,M) - h^{n-1}(G,M) + \dots
+ (-1)^n h^0(G,M)\right )$$
where $K$ is a field, $M$ is a $KG$-module,  and $h^n(G;M) := \dim_KH^n(G; M)$. Moreover we can
assume that $K$ is algebraically closed and has characteristic $p$ dividing $|G|$, and it suffices to verify the upper bound on absolutely irreducible modules. 

We now focus on an interesting class of non--abelian groups for which the absolutely irreducible modules are easy to determine. The following proposition follows from
elementary representation theory (see \cite[Corollary 6.2.2]{Webb:2016})

\begin{proposition}
Let $K$ be a field of characteristic $p$, $G$ a finite group with maximal normal $p$--group denoted by $O_p(G)$.
Then the simple $KG$--modules are precisely the simple $K[G/O_p(G)]$--modules, made into $KG$--modules
via the quotient homomorphism $G\to G/O_p(G)$. 
\end{proposition}

\begin{corollary}
Let $U_k=E_k\times_TC$ denote a semi-direct product, where $E_k\cong (\bZ/p\bZ)^k$ with 
$k>1$ and $C$ is cyclic of order
relatively prime to $p$. Then for any algebraically closed field $\bK_p$ of characteristic $p$, 
the absolutely irreducible $\bK_pU_k$--modules are
one dimensional characters $\alpha\colon C \to \bK_p^{\times}$ on which $E_k$ acts trivially. 
\end{corollary}

The cyclic group $C$ acts on the vector spaces $H^i(E_k; \bK_p)$ via one-dimensional characters
$\alpha\colon C \to \bK_p^{\times}$. 
Using the multiplicative structure in cohomology and the Bockstein, this is determined by $N_k=H^1(E_k;\bK_p)\cong \Hom (E_k, \bK_p)$ as an $\bK_p[C]$--module.
 
 Recall that by
\cite[Corollary II.4.3, Theorem II.4.4]{Adem:2004}, the mod $p$ cohomology ring of $E_k$ is given by
$$H^*(E_k, \bF_p)\cong
\begin{cases}
\bF_2 [x_1, \dots , x_k] & \text{for $p=2$}\\
\Lambda (x_1, \dots, x_k)\otimes \bF_p [y_1, \dots, y_k]  & \text{for $p$ odd}
\end{cases}
$$
where $x_1, \dots , x_k\in H^1(E_k, \bF_p)$, $y_1,\dots ,y_k\in H^2(E_k, \bF_p)$ and 
$\Lambda (x_1, \dots, x_k)$ denotes the exterior algebra on these one-dimensional generators. Moreover
if we let $B\colon H^1(E_k, \bF_p)\to H^2(E_k, \bF_p)$ denote the Bockstein, then we can assume that for 
$p$ odd $B(x_i)= y_i$, whereas for $p=2$, $B(x_i)=x_i^2$ for all $i=1, \dots , k$. By extending coefficients we obtain the same structure for $H^*(E_k, \bK_p)$.

The map $B$ is compatible with respect to the $C$ action and defines an isomorphism onto its image,
thus giving rise to
an exact sequence
$$0\to N_k\to H^2(E_k;\bK_p)\to \Lambda^2(N_k)\to 0$$ 
as 
$\bK_pC$--modules. 
If $N_k\cong \bigoplus_{1\le i\le k} L(\alpha_i)$ then 
$\Lambda^2(N_k)\cong \bigoplus_{1\le i<j\le k} L(\alpha_i\alpha_j)$. Note that if we tensor the sequence
with any other character and take $C$--invariants it will still be exact, as $(|C|,p)=1$.  

Using the fact that for any $\bK_pU_k$--module $M$, $H^t(U_k;  M) \cong H^t(E_k;  M_{|_E})^C$ for every $t\ge 0$, for any character
$L(\beta)$ we obtain the formula
$$ h^2(U_k;L(\beta)) - h^1(U_k;L(\beta)) + h^0(U_k;L(\beta)) =
\dim [\Lambda^2(N_k)\otimes\beta]^C + \dim L(\beta)^C.$$ 
At primes $q$ dividing $ |C|$,  we work over the field $\bK_q$, and note that $H^i(U_k, L)= H^i(C, L^{E_k})$. Hence an absolutely irreducible $L$ with some
$h^i(U_k,L)\ne 0$ must also have a trivial action of $E_k$. Arguing as before, $L$ is the inflation of a character
$C/O_q(C)\to \bK_q^{\times}$. Thus we have
$H^i(U_k, L)= [H^i(O_q(C), \bK_q)\otimes L]^{C/O_q(C)}$. As $O_q(C)$ is cyclic, all these terms are isomorphic and of 
non-zero rank (equal to one) if and only if the action of $C/O_q(C)$ is trivial and we obtain that
$$h^2(U_k;L) - h^1(U_k;L) + h^0(U_k;L)\le 1.$$
We apply our analysis to obtain a calculation for $\mu_2(U_k)$ and $e_2(U_k)$: 

\begin{proposition}\label{prop:bound}
For $U_k = E_k\times_T C$ as above, with $N_k=H^1(E_k;  \bK_p)$,  
$$\mu_2(U_k)=\max \{\dim [\Lambda^2(N_k)\otimes L(\beta)]^C + \dim L(\beta)^C\}$$
as $L(\beta)$ ranges over all characters $\beta\colon C \to \bK_p^{\times}$, and
$$e_2(U_k) = \dim \Lambda^2(N_k)^C  - \dim N_k^C + 2.$$
\end{proposition}

We apply this to the special case when $p$ is an odd prime and the action on $N_k=H^1(E_k, \bK_p)$ is \textsl{isotypic}, i.e. it 
is the direct sum of copies of a fixed character $L(\alpha)$. 

\begin{corollary}\label{cor:bound}
Let $U_k = E_k\times_T C$ where $p$ is odd and the action of C on the vector space $E_k$ gives rise to
the sum of $k$ copies
of a fixed character $L(\alpha)$ over the splitting field $\bK_p$, with $k>1$.
\begin{enumerate}
\setlength\itemsep{8pt}
\item
If $\alpha^2\ne 1$, $e_2(U_k)=2$, $\mu_2(U_k) = \frac{k(k-1)}{2}$ and $\mu_2(U_k)-\mu_1(U_k) = \frac{k(k-3)}{2}.$
\item
If $\alpha^2=1$, $e_2(U_k) = \frac{k(k-1)}{2} + 2$, $\mu_2(U_k) = \frac{k(k-1)}{2} + 1$ and $\mu_2(U_k) - \mu_1(U_k) = \frac{k(k-3)}{2} + 1.$ 
\end{enumerate}
\end{corollary}
\begin{proof} We apply Proposition \ref{prop:bound} to compute
$\mu_2(U_k)$ and $e_2(U_k)$.  
Choose $\beta = \alpha^{-2}$, then $\Lambda^2(H^1(E_k;\bK_p))\otimes L(\beta)$ is a trivial $\bK_p[C]$--module
of dimension equal to $\frac{k(k-1)}{2}$. In the special case $\beta =1$ we obtain the extra term. The calculation for
$e_2(U_k)$ follows from its expression in terms of invariants.
As $U_k$ is solvable we have $\mu_1(U_k) =k$ and the proof is complete. 
\end{proof}

\begin{corollary}\label{cor:q4-uk}
Let $U_k=E_k\times_T C$ where $p$ is odd, $E_k=(\bZ/p\bZ)^k$ and $C$ cyclic of order prime to $p$ acts on each $\bZ/p\bZ$ factor in $E_k$ via $x\mapsto x^q$ where 
$q$ is a unit in $\bZ/p\bZ$. 
\begin{enumerate}
\item If $x^{q^2}\ne x$ for all $1\ne x\in E_k$, then 
$$\max\{2, \frac{k(k-3)}{2}\} \le q_4(U_k) \le k(k-1)$$
\item If $q=p-1$, then
$$\frac{k(k-1)}{2} + 2 \le q_4(U_k)\le k(k-1) + 2$$
\end{enumerate}
\end{corollary}

\begin{remark} 
Note that for $\alpha^2\ne1$, $e_2(U_k)=2 < \mu_2(U_k)-\mu_1(U_k)$ for $k\ge 5$. Hence \ref{cor:q4-uk}
improves on the lower
bound given in \cite{Hausmann:1985}. On the other hand, if $\alpha^2 =1$ then $\mu_2(U_k) - \mu_1(U_k) < e_2(U_k)$ for all
$k>1$. This shows that the two invariants play a role in establishing lower bounds for $q_4(G)$. 
\end{remark} 
We now apply our estimate to homology 4--spheres. 
\begin{theorem}\label{thm:better-bound}
Let $U_k=E_k\times_T C$ where $p$ is an odd prime, $E_k=(\bZ/p\bZ)^k$ and $C$ cyclic of order prime to $p$ acts on each $\bZ/p\bZ$ factor in $E_k$ via $x\mapsto x^q$ where 
$q$ is a unit in $\bZ/p\bZ$. 
\begin{enumerate} 
\item If $x^{q^2}\ne x$ for all $1\ne x\in E_k$, then for all $k>4$, $U_k$ does not arise as the
fundamental group of any rational homology 4--sphere.
\item If $q=p-1$, then for all $k>1$, $U_k$ does not arise as the
fundamental group of any rational homology 4--sphere. 
\end{enumerate}
\end{theorem}
\begin{proof}
For the groups $U_k= E_k\times_TC$ where $x^{q^2}\ne x$ for $x\in E_k$, we consider the inequality
$$\max\{2, \frac{k(k-3)}{2}\} \le q_4(U_k) \le k(k-1).$$
If a rational homology 4--sphere $X$ with fundamental group $U_k$ exists, then $q_4(U_k)=2$ and we'd have
$\frac{k(k-3)}{2} \le 2$, which implies that $k\le 4$. Note that the upper bound implies the existence of a rational
homology 4--sphere with fundamental group $U_2$.
In the case when $q=p-1$ our estimate \ref{cor:q4-uk} shows that $2< q_4(U_k)$ for all $k>1$. 
\end{proof}

\begin{example}\label{ex:fivenine}
Let $\bF_q$ denote a field with $q=2^k$ elements. Then the cyclic group of units $C=\bZ/(q-1)\bZ$ acts transitively on
the non--zero elements of the underlying mod 2 vector space $E_k=(\bZ/2\bZ)^k$. If we write
$\bF_q = \bF_2[u]/(p(u))$, where $p(u)$ is an irreducible polynomial of degree $k$ over $\bF_2$, then the action can
be described as multiplication by $u$. Expressing it in terms of the basis $\{1,u,\dots,u^{k-1}\}$ we obtain a faithful representation
$C\to GL(k,\bF_2)$ with characteristic polynomial $p(t)$. This gives 
rise to a semi--direct product $J_k= E_k\times_TC$ where the
action of $C$ on $N_k=H^1(E_k,\bK_2)$ decomposes into non-trivial, distinct characters determined by the roots of
$p(t)$. If $\alpha$ is a root of this
polynomial,  so are all the powers $\{\alpha^{2^i}\}_{i=0,\dots,k-1}$ and these appear as a complete set of eigenvalues for
the action on the $k$--dimensional vector space. In other words we have
$N_k\cong \bigoplus_{0\le i\le k-1} L(\alpha^{2^i})$.
We propose to compute the invariants $\mu_2(J_k)$ and $e_2(J_k)$. 
\begin{proposition}
For the groups $J_k$ described above we have
\begin{enumerate}
\item
$\mu_2(J_2) = 2$, whereas $\mu_2(J_k)=1$ for all $k>2$. 
\item
$e_2(J_2)=3$, whereas $e_2(J_k)=2$ for all $k>2$.
\end{enumerate}
\end{proposition}
\begin{proof} As $N_k\cong \bigoplus_{0\le i\le k-1} L(\alpha^{2^i})$ we have that
$\Lambda^2(N_k)\cong \bigoplus_{0\le i<j\le k-1} L(\alpha^{2^i+2^j})$. 
For $k>2$, this
is a sum of distinct, non--trivial characters. This follows from the fact that for
$k>2$, 
$$2^k-1 = 2^{k-1}+2^{k-2} +\dots + 2+1>2^i+2^j,$$
and each $\alpha^{2^i+2^j}$ is a distinct,
non-trivial $2^k-1$ root of unity.
Hence if $k>2$, the module $\Lambda^2(N_k)$ has 
\textbf{no
trivial summands}, and \textbf{no repeated summands}. 
Now if we take any character $L(\beta)$, we see that at most one summand in 
$\Lambda^2(N_k)\otimes L(\beta)$ can be trivial. And if this occurs then $\beta \ne 1$. Hence 
applying the formula 
in Proposition \ref{prop:bound} we conclude
that $\mu_2(J_k)=1$.
Similarly we see that $e_2(J_k)=2$
for all $k>2$. Also $\Lambda^2(N_2)\cong
L(1)$, whence we see that $\mu_2(J_2)=2$ and $e_2(J_2)=3$.
\end{proof} 

From these examples we conclude that there exist rational homology 4--spheres with fundamental group 
equal to $J_k$ for
$k>2$, and so groups of arbitrarily high rank can occur as such groups, in contrast to the situation for abelian groups appearing in Corollary \ref{cor:twoseven}. 

For $J_2\cong A_4$, the alternating group on four letters, we have $e_2(J_2)=3$, $\mu_2(J_2)=2$ and so 
$3\le q_4( A_4)\le 4$. The cohomological computations also imply that $\mu_4( A_4)=1$, whence there does exist a rational homology $8$--sphere with fundamental group $ A_4$. 
\end{example}
\begin{proposition} For the alternating groups $G = A_4$ or $G=A_5$, we have $q_4(G) = 4$.  
\end{proposition}
\begin{proof} By the estimates above,  for $G= A_4$ we only need to rule out $q_4(G) = 3$, so suppose that there exists $M^4$ with $\pi_1(M) = A_4$ and $\chi(M) = 3$. Applying the universal coefficient theorem, we can use the computation at $p=2$ 
(see \cite[Theorem 1.3, Chapter III]{Adem:2004}) to show that $H_4(A_4;\bZ) = 0$.
Hence $\wH^{-5}(A_4,\bZ)=0$ and applying 
Proposition \ref{prop:exponent} we infer that  
$ \pi_2(M)$ is stably isomorphic to $J \oplus J^*$, where $J$ denotes a minimal representative of $\Omega^3(\bZ)$.  Since $\chi(M) = 3$, we have $H^0(G; \pi_2(M)) = \bZ$. From the exact sequence for Tate cohomology
 \cite[Chap.~IV.4]{Brown:1982},
we have a surjection
$\bZ = H^0(G;\pi_2(M)) \twoheadrightarrow \wH^0(G;\pi_2(M))$.
However,
$$\wH^0(G;\pi_2(M)) =  \wH^0(G;J \oplus J^*) = \wH^{-3}(G;\bZ)\oplus \wH^{3}(G;\bZ) \cong H_2(G;\bZ) \oplus  H_2(G;\bZ)$$
and since $H_2(G;\bZ) = \cy 2$, this is impossible. 

For $G= A_5$ we apply the fact that for every non-periodic finite subgroup $G$ of $SO(3)$, $\mu_2(G)=2$  (see Remark \ref{subgroups-so3}). The rest of the argument is analogous to that for $A_4$, since the restriction map $H^*(A_5;\bZ/2\bZ) \to H^*(A_4; \bZ/2\bZ)$ is an isomorphism. This is true
because both groups share the same $2$--Sylow subgroup $(\bZ/2\bZ)^2$, with normalizer $A_4$ (see \cite[Theorem 6.8, Chapter II]{Adem:2004}). This implies  that $e_2(A_5)=3$, $H_4(A_5;\bZ) = 0$ and $H_2(A_5;\bZ) = \cy 2$ (note that the
other two $p$--Sylow subgroups are cyclic, so don't contribute to even degree homology). Therefore we can rule
out $q_4(A_5)=3$ whence $q_4(A_5)=4$.
\end{proof}

\section{Some further remarks and questions in dimension four}\label{sec:five}

In this section we will briefly discuss some  questions about rational homology $4$-spheres whose  fundamental groups are finite.

\medskip
\noindent\textbf{\S 6A. Existence via Surgery.} The main open problem is to characterize the finite groups $G$ for which $q_4(G)=2$. To make progress,  we need more constructions of rational homology $4$-spheres.

Examples of $\bQ S^4$-manifolds can be constructed by starting
 with a rational homology $3$-sphere $X$, forming the product $X \times S^1$, and then doing surgery on
 an embedded $S^1 \times D^3 \subset X \times S^1$ representing a generator of $\pi_1(S^1) = \bZ$. This construction is equivalent to the ``thickened double" construction $Z = M(K)$ for a finite $2$-complex of  Proposition \ref{prop:existence} (compare \cite[\S 4]{hk2}).
 
 \begin{example} The groups $G=\cy p \times \cy p$ are $\bQ S^4$-groups, since we can do surgery on an embedded circle $L^3(p,1) \times S^1$ representing $p$-times a generator of $\pi_1(S^1) = \bZ$.   These examples are \emph{not } of the thickened double form $M(K)$ because the minimal rank of $\pi_2(K)$ representing $\Omega^3(\bZ)$  is greater than $|G|-1$, and hence  the extension describing $\pi_2(M)$ is non-trivial (by Proposition \ref{prop:extclass}).
 \end{example}

 Since the quotient of a free finite group action on a rational homology $3$-sphere is again a rational homology $3$-sphere, 
 one could use the examples $X=Y/G$ studied by \cite{Adem:2019}, where $Y$ is a $\bQ S^3$ and $G$ is a finite group acting
 freely on $Y$. However, to obtain a $\bQ S^4$ with finite fundamental group by this construction, $Y$ must itself have finite 
 fundamental group.

\begin{remark}
The finite fundamental groups of closed, oriented $3$-manifolds have periodic cohomology of period 4, but not all $4$-periodic groups arise this way. 
A complete list of  $4$-periodic groups is given in Milnor \cite[\S 3]{milnor2}, and those which can act freely and orthogonally on $S^3$ were listed by Hopf  \cite{Hopf:1926}. Perelman \cite{Lott:2007} showed that the remaining groups in Milnor's list do not arise as the fundamental group of any closed, oriented $3$-manifold, and that the closed $3$-manifolds with finite fundamental group are exactly the $3$-dimensional spherical space forms.
\end{remark}

\begin{remark}\label{subgroups-so3} For every non-periodic finite subgroup $G$ of $SO(3)$, 
 we have $\mu_2(G)=2$ and hence $q_4(G) \leq 4$ (see \cite[Proposition 2.4]{hk4}). Note that each such subgroup has a $2$-fold central  extension $G^* \subset SU(2)$  which acts freely on $S^3$, and let $X =S^3/G^*$ denote the quotient $3$-manifold.  On $N:=X \times S^1$, we can do surgery on disjoint circles representing (i) a generator of the central subgroup of $G^*$ and  (ii) a generator of $\bZ$,  to reduce the fundamental group from $\pi_1(N) = G^* \times \bZ$ to $G$. We thus obtain a $4$-manifold $M$ with $\chi(M)=4$ and $\pi_1(M) = G$, realizing the upper bound for $q_4(G)$. Our estimates give $2 \leq q_4(G) \leq 4$ for the cases not yet determined,  namely where  $G$ is dihedral of order $4n$ or $G$ is the symmetric group  $S_4$.
\end{remark}

\begin{remark} Teichner \cite[3.7]{Teichner:1988}  indicated that topological surgery could produce examples with finite fundamental group from certain $4$-manifolds with infinite fundamental group.  This technique should be investigated further.
 \end{remark}

\noindent\textbf{\S 6B. Groups of Deficiency Zero.} There are many finite groups with  deficiency zero:  for  example,
Wamsley \cite{Wamsley:1970} 
showed that a metacyclic group $G$ with $H_2(G;\bZ) =0$ has $\Def(G) = 0$. In particular, the class of finite
groups arising as fundamental groups of rational homology 4--spheres
includes groups with periodic cohomology of arbitrarily high period. There is an extensive literature on this problem: for example, see \cite{Cossey:1974,Campbell:1980,Campbell:1980a,Epstein:1961,Johnson:1970,Mennicke:1989,Neumann:1989,Sag:1973,Sag:1973a,Searby:1972}.

According to Swan, $1\leq \mu_2(G) \leq 1 - \Def(G)$ 
(see \cite[Proposition 1, Corollary 1.3]{Swan:1965}), hence if $G$ is a  finite group of deficiency zero, we have
$\mu_2(G)=1$. Thus for such groups by Proposition \ref{prop:existence}, we can construct an orientable 4--manifold $M$ with $\pi_1(M)=G$ and
$\chi(M)=2$.  More generally, this can be done whenever $\mu_2(G) =1$ by Theorem B (see Corollary \ref{cor:upper} and the series of groups $J_k$ considered in Example \ref{ex:fivenine}).  
Then $M$ is a rational homology 4--sphere, and in these cases there is a  minimal representative $J$ for  the stable module $\Omega^3(\bZ)$ with 
$\rk_{\bZ}(J) = |G|-1$. (compare \cite[Corollary 4.4]{hk2}). For example, if $G$ is the fundamental group of a closed, oriented $3$-manifold, then $J \cong I(G)^*$. 

\begin{remark}
We are indebted to Mike Newman and \"Ozg\"un \"Unl\"u for showing that some of the groups $J_k$ do have deficiency zero (e.g.~at least for $3 \leq k \leq 6$).  It is a challenging open problem to decide whether this is true for all $k \geq 3$. Note that any group in this range which does not admit a balanced presentation would give a negative answer  to Wall's D2 problem.
\end{remark}

\begin{example} 
 Teichner \cite[3.4, 4.15]{Teichner:1988} proved that if $G$ is a finite $\bQ S^4$-group, then $d(H_1(G)) \leq 7$, and used a mapping torus construction to produce a non-abelian $\bQ S^4$-group $G$ with $d(H_1(G)) = 4$. 
\end{example}

\noindent\textbf{\S 6C. Algebraic Questions.}
For the rational homology $4$-spheres $M$ with $\pi_1(M)=G$ constructed in Theorem B, we have
$\pi_2(M) = H_2(\widetilde M;\bZ) = J \oplus J^*$, where $J$ is a minimal representative for $\Omega^3(\bZ)$ over $\ZG$,  with $\rk_{\bZ}(J) = |G|-1$. 

 Moreover, $J$ is locally and hence rationally isomorphic to the augmentation ideal $I(G)$, and the equivariant intersection form $s_M$ on $\pi_2(M) = J \oplus J^*$ is metabolic, with totally isotropic submodule $0 \oplus J^*$. Similar results hold for the higher dimenaional examples constructed in  Proposition \ref{prop:existence}. 
 
 More generally, for any finite group $G$, the existence of a   representative $J$ for  the stable module
 $\Omega^3(\bZ)$ with 
$\rk_{\bZ}(J) = |G|-1$ is equivalent to the condition $\mu_2(G) =1$ (see Proposition \ref{prop:threetwo}).

  \begin{question}  Is there a finite group $G$ with $\mu_2(G) = 1$, such that $G$ is  neither periodic nor admits a balanced presentation ?
 \end{question} 
 For any closed oriented $4$-manifold $M$ with finite fundamental group $G$, we have seen in \ref{prop:extclass}
 that $\pi_2(M)$ is \emph{stably} given by an extension of $\Omega^{-3}(\bZ)$ by $\Omega^3(\bZ)$ (see also \cite[Proposition 2.4]{hk2}) and that the extension class in 
  $\Ext^1_{\ZG}(\Omega^{-3}(\bZ), \Omega^{3}(\bZ)) \cong H_4(G; \bZ)$ is given by the image of the fundamental  class of $M$.
  For any rational homology $4$-sphere $M$
 with finite fundamental group $G$, the condition $\chi(M) =2$ implies that $\rk _{\bZ} (\pi_2(M)) = 2(|G| -1)$ and $H^0(G; \pi_2(M)) = 0$.

\begin{question}  If $M$ is a $\bQ S^4$, what is the (unstable)  structure of $\pi_2(M)$ as an integral representation ? Is the equivariant intersection form $s_M$ always metabolic (in the sense defined in \cite[\S 2]{hk6}) ?
\end{question}

Finally we point out that many questions in the representation theory of finite groups can be investigated by induction and restriction to proper subgroups. At present we do not see how to apply this technique in our setting.

\begin{question}  If $M$ is a $\bQ S^4$ -manifold with finite fundamental group $G$, then  its non-trivial finite coverings have  Euler characteristic $> 2$ (and hence are not $\bQ S^4$-manifolds).  How can we decide if  proper subgroups of $G$ are also $\bQ S^4$-groups ?
\end{question}

\section{Appendix A: The Proof of  Theorem \ref{prop:fakeembed} } \label{sec:seven}
In this section we give a direct construction of the minimal $4$-manifold needed for Theorem \ref{prop:fakeembed}. The idea is to use a handlebody thickening (see Definition \ref{def:seventwo}) of a finite $2$-complex $K$ instead of starting with an embedding of $K$ in $\bbR^5$. The advantage of this thickening is that we can identify the intersection form of its $4$-manifold boundary, and then apply a recent refinement of Freedman's work due to Teichner, Powell and Ray (see \cite[Corollary 1.4]{Powell:2020}).

\medskip
\noindent\textbf{\S 7A. Metabolic Forms.}\label{subsec:sevena}
To analyse the intersection form of the handlebody thickening we will need some algebraic preparations.

\begin{definition}
 Let $(E, [q])$ denote a  quadratic metabolic  form on a $\La$-module $E = N \oplus N^*$, where $N$ is a left $\La$-module, and  $N^*$ inherits a left $\La$-module structure via the standard anti-involution $ a \mapsto \bar a$ on $\La = \ZG$. Then
$$q((x, \phi), (x', \phi')) = \phi(x') + g(\phi, \phi'),$$
where $x, x' \in N$, $\phi, \phi' \in N^*$ and $g \in \Hom(N^* \otimes N^*, \La)$ is a sesquilinear form. We use the notation $(E, [q]) = \Met(N, g)$ for this metabolic form (see \cite[\S 2]{hk6} for  metabolic forms defined on a non-split extension of  $N$ and $N^*$).
\end{definition}

The associated hermitian form $h = q + q^*$ is non-singular, 
and $N \oplus 0 \subset E$ is a totally isotropic direct summand. More explicitly,
$$h((x, \phi), (x', \phi')) = \phi(x') +  \overline{\phi'(x)} + g(\phi, \phi') + \overline{g(\phi', \phi)}.$$
In our geometric setting, the metabolic forms arise on modules 
$E = H^2(K) \oplus H_2(K)$, where $K$ is a finite $2$-complex with fundamental group $G$ (take coefficients in $\La = \ZG$). If $G$ is finite, then $H^2(K; \La) \cong \Hom_\La(H_2(K), \La)$, and the definition above applies. If $G$ is infinite, then we slightly generalize our notion of metabolic form. 

\begin{definition}\label{def:sevenone}
Let $E = N \oplus \wN$, and let $\alpha\colon \wN\to N^*$ be a $\La$-module homomorphism. Define a \emph{generalized} metabolic form $(E, [q]):= \Met(N, \wN, \alpha, g)$ by the formula
$$q((x, \phi), (x', \phi')) = \alpha(\phi)(x') + g( \phi,  \phi'),$$
where $x, x' \in N$, $\phi, \phi' \in \wN$, and $g \in \Hom(\wN \otimes \wN, \La)$ is a given sesquilinear form. 
\end{definition}
\begin{example}\label{ex:sevenone}
For a finite $2$-complex $K$,  we have the evaluation map $\alpha\colon H^2(K) \to \Hom_\La(H_2(K), \La)$, which in general is neither injective nor surjective. In this case, we will  shorten the notation of Definition \ref{def:sevenone}  to  $(E, [q]) =\Met(H_2(K), g)$, where $E = H_2(K) \oplus H^2(K)$  as above.
\end{example}

Here are some preliminary remarks. 
\begin{itemize}
\item  Let $(E, [q])$ be  any quadratic form, and suppose that $U$ is finitely generated  submodule  on which the restriction $\lambda_0$ of $\lambda = q + q*$ to $U$ is non-singular. Then there is  is orthogonal splitting $(E, [q]) \cong U \perp L$. 
\begin{proof} Consider the following sequence
$$ 0 \to U \to E \xrightarrow{\adj \lambda} E^* \to U^* \to 0$$
where the composition $\adj \lambda_0\colon U \to U^*$ is an isomorphism by assumption. Therefore the inclusion $U \subset E$ is a split injection, and $E = U \perp L$, where $L := U^\perp$. To check this last point, note that a splitting map for the inclusion $i \colon U \to E$ is given by 
$$r := (\adj \lambda_0)^{-1} \circ i^* \circ \adj \lambda \circ i.$$
For $e \in E$, we compute 
$$\hphantom{xxxxx} \lambda(e - i(r(e)), i(h)) = \lambda (e, i(h)) - \lambda (i(r(e)), i(h)) = 
\lambda (e, i(h))  - \lambda_0(r(e), h) = 0$$
after substituting the formula for $r$. Therefore $E = U +  U^\perp$, and $U \cap U^\perp = 0$ since $\lambda_0$ is non-singular. 
\end{proof}
\item Let $(E, [q]) = \Met(H,g)$ be a metabolic quadratic form on $E = H\oplus H^*$, where $H = \La^r$ is a finitely generated free $\La$-module. Then $(E, [q]) \cong H(\La^r)$ is a hyperbolic form.
\begin{proof} This is a standard fact (see \cite[Lemma 5.3]{wall-book}).
\end{proof}
\end{itemize}
  
\noindent\textbf{\S 7B. A Handlebody Thickening.}
  Let $K$ be a finite $2$-complex with  $\pi_1(K) = G$. We construct a suitable thickening of $K$ to be used in the proof of Theorem \ref{prop:fakeembed}.

  \begin{definition}\label{def:seventwo}
 We first consider a $4$-dimensional parallelizable thickening $A(K)$ of $K$ constructed by attaching suitable $2$-handles to a connected sum $\Sharp \ell(S^1\times D^3)$. Then $A(K)$ is a compact $4$-manifold with boundary, and we let  $N(K) = A(K) \times I$. Note that $N(K)$ is a $5$-dimensional thickening of $K$, but may not embed in $\bbR^5$, and that $\bd N(K) = A(K) \cup -A(K)$ is the double of $A$ along the common boundary.
 \end{definition}
 
  Then  $ M:= \bd N(K)$ has the  intersection form $\lambda_M =\Met(H_2(K), g)$, since 
  $H_2(\bd N(K)) = H^2(K) \oplus H_2(K)$ and the direct summand $H^2(K)$ is totally isotropic (compare \cite[\S 2]{kreck-schafer1}). All the homology groups have coefficients in $\La := \ZG$.
  
  \begin{remark} Note that the quadratic intersection form  $\Met(H_2(K), g)$ is  a \emph{generalized }metabolic form (see  Example \ref{ex:sevenone}). It is non-singular if $\pi_1(K) = G$ is a finite group. If $G$ is infinite, this form has radical $H^2(G; \La)$, and the cokernel of its adjoint is $H^3(G; \La)$  by the exact sequence
  $$0 \to H^2(G;\La) \to H^2(M;\La) \xrightarrow{\adj \lambda_M} \Hom_{\La}(H_2(M), \La) \to H^3(G;\La) \to 0$$ 
  arising from the universal coefficient theorem.
  \end{remark}

\noindent\textbf{\S 7C. A Self-Homotopy Equivalence.} 
   Let $N(K)_r := N(K) \natural\, r(S^2 \times D^3)$ denote this new thickening of $K \vee r(S^2)$. We recall the construction of a useful homotopy self-equivalence of $K \vee r(S^2)$. 
   
   \begin{lemma}[\rm{\cite[Lemma 2.1]{Hambleton:2019}}] Let $X$ be a finite \fake, and let $u\colon K \subset X$ denote the $2$-skeleton of $X$. Then, for $r = b_3(X)$ there is a simple self-homotopy equivalence
    $h \colon K \vee r(S^2) \to K \vee r(S^2)$ inducing a simple homotopy equivalence  $f\colon X \vee r(S^2) \simeq K$.
   \end{lemma} 
   \begin{proof}
We recall some  of the notation from \cite[\S 2]{Hambleton:2019}. There is an identification 
\eqncount
\begin{equation}\label{eq:twotwo}
 \pi_2(K \vee r(S^2)) \cong \pi_2(K) \oplus \La^r\cong \pi_2(X) \oplus C_3(X) \oplus F
 \end{equation}
and we fix free $\La$-bases $\{e_1, \dots, e_r\}$  for $C_3(X) \cong \La^r$, and  $\{f_1, \dots, f_r\}$ for $F \cong \La^r$. 
The same notation $\{e_i\}$ and $\{f_j\}$ will also be used for continuous maps $S^2 \to K \vee r(S^2)$ in the homotopy classes 
of $ \pi_2(K \vee r(S^2))$ defined by these basis elements. Notice that the maps $f_j\colon S^2 \to K \vee r(S^2)$ may be chosen
to represent the inclusions of the $S^2$ wedge factors.

An examination of the proof of  \cite[Lemma 2.1]{Hambleton:2019} shows that the simple homotopy equivalence 
 $f\colon X \vee r(S^2) \simeq K$  is obtained by extending a certain simple homotopy equivalence 
$h \colon K \vee r(S^2) \to K \vee r(S^2)$ over the (stabilized) inclusion
$$\xymatrix{  X \vee r(S^2) \ar@{-->}[rr]^{h'} & &K' \\  K \vee r(S^2)\ar@{^{(}->}[u]_{u \vee \id} \ar[rr]^h& &  K \vee r(S^2) \ar@{^{(}->}[u]
}$$  
 by attaching the $3$-cells of X  in domain by the  maps $e_i =[\bd D^3_i]$, $1\leq i \leq r$, and  $3$-cells in the range via the maps $f_i =[\bd D^3_i]$, $1\leq i \leq r$, which homotopically cancel the $S^2$ wedge factors, to obtain a complex $K' \simeq K$.

 Then  we have $h \circ [\bd D^3_i ]= f_i$ by the construction of $h$ (see \cite[p.~364]{Hambleton:2019}).  Hence we can extend  $h$ over $X$ by the identity on the $3$-cells attached  in domain and range along the maps $\{f_i\colon S^2 \to  K \vee r(S^2)\}$. We obtain
a map
$$h'\colon X\vee r(S^2)  \to K' := K \vee r(S^2)\cup \bigcup\{ D^3_i : [\bd D^3_i] = f_i, 1\leq i \leq r\} $$
extending $h$. From the construction of the map $h$ (see 
\cite[p.~364]{Hambleton:2019})  follows that $h'$ is a (simple) homotopy equivalence, which induces a simple homotopy equivalence  $f\colon X \vee r(S^2) \simeq K$, after composition with the obvious projection $K' \to K$.
   \end{proof}
   
\noindent\textbf{\S 7D. Topological Surgery.} 
  We will now apply some results of topological surgery due to Freedman. Recall that  $N(K)$ is the $5$-dimensional thickening of $K$ constructed above,  and $N(K)_r = N(K) \natural\, r(S^2 \times D^3)$ is its stabilization. We have introduced the notation $M = \bd N(K)$, and let $M_r = \bd N(K)_r = \bd N(K) \# r(S^2 \times S^2)$,
  \begin{lemma}\label{lem:seveneight}  Suppose that $\pi_1(K)$ is a good group. There is a self-homeomorphism $\beta\colon \bd N(K)_r \approx \bd N(K)_r$ extending the simple homotopy self-equivalence
  $h \colon K \vee r(S^2) \to K \vee r(S^2)$.
  \end{lemma} 
  \begin{proof} 
    Since $\pi_1(K)$ is a good group, the topological $s$-cobordism theorem \cite[Theorem 7.1A]{freedman-quinn1}, implies that the given simple homotopy self-equivalence $h \colon K \vee r(S^2) \to K \vee r(S^2)$ extends to a  self-homeomorphism $\hat h \colon N(K)_r \to N(K)_r$. This follows since we may assume (by general position) that  the image $h( K \vee r(S^2)) \subset N(K))_r$ is embedded in the interior of the $5$-manifold  $ N(K))_r$. Since $h$ is a simple homotopy self-equivalence, the complement of a small tubular neighbourhood of  $h( K \vee r(S^2))$ will then be an $s$-cobordism, and hence a product. Since $ N(K))_r$ is a thickening of $K \vee r(S^2)$, we can construct the self-homeomorphism $\hat h$ by identifying the tubular neighbourhoods in domain and range, and then using the product structures. Let $\beta:=\bd \hat h$ denote the restriction of $\hat h$ to $\bd N(K)_r$.
    \end{proof}

  We now combine these ingredients. Recall that $X \vee r(S^2) \simeq K$, so that $H_2(K) \cong H_2(X)\oplus H$, where $H \cong \La^r$.
 We have the isomorphism
  \eqncount
 \begin{equation}\label{eq:sevenone}
 H_2(N(K)_r) = H_2(K \vee r(S^2)) \cong H_2(K) \oplus F\cong H_2(X) \oplus H \oplus F,
 \end{equation}
where $F\cong \La^r$. We fix free $\La$-bases $\{e_1, \dots, e_r\}$  for $H \cong \La^r$, and  $\{f_1, \dots, f_r\}$ for $F \cong \La^r$. It follows that $M_r := \bd N(K)_r$ has intersection form 
$$\lambda_{M_r} = \lambda_{M} \oplus H(F) = \Met(H_2(K), g) \oplus H(\La^r),$$
 where the classes
 $\{f_1, f_2, \dots, f_r\}$ and their duals provide a standard hyperbolic base for the
  second summand $H(F)$.
 By construction, $h_*(e_i) = f_i$, $h_*(f_i) = e_i$ for $1 \leq i \leq r$, and $h_*(x) = x$ for all $x \in H_2(X)$.
 Note that $H^2(K) = H^2(X)\oplus H^*$ is totally isotropic under $\lambda_M$, 
 and orthogonal to the summand $H(F)$.

\begin{lemma} There is closed topological $4$-manifold $M_0$ and a homeomorphism 
$M_r = \bd N(K)_r \approx M_0 \Sharp 2r(S^2 \times S^2)$, such that $\chi(M_0) = 2 \chi(X)$.
\end{lemma}
\begin{proof}  We have the decomposition:
$$H_2(M_r) = H^2(X) \oplus H_2(X) \oplus H \oplus H^*\oplus F\oplus F^*$$ in the notation introduced in \eqref{eq:sevenone}.

The metabolic intersection form $\lambda_{M_r} = \lambda_{M} \oplus H(F)$  admits a self-isometry $\beta_*$ (induced from the map $\beta$ constructed in Lemma \ref{lem:seveneight}) extending the map $h_*\colon H_2(K) \oplus \La^r \to H_2(K) \oplus \La^r$ constructed above. Since the images of the basis elements $h_*(e_i) = f_i \in F$ have dual classes $f^*_i \in 
F^*$, it follows that 
\eqncount
\begin{equation}\label{eq:seventwo}
 \lambda_{M_r}(\beta_*(f^*_i), e_j) = 
\lambda_{M_r}(\beta_*(f^*_i), \beta_*(f_j)) =
 \lambda_{M_r}(f^*_i, f_j) = \delta_{ij}.
 \end{equation}
 Similarly, we have the formulas
  \eqncount
 \begin{equation}\label{eq:seventhree}
  \lambda_{M_r}(\beta_*(f^*_i), \beta_*(f^*_j)) = 
 \lambda_{M_r}(f^*_i, f^*_j) = 0, \text{\ for all\ } 1\leq i, j \leq r,
 \end{equation}
 and 
 \eqncount
\begin{equation}\label{eq:sevenfour}
\lambda_{M_r}(e_i, e_j) =  \lambda_{M_r}(\beta_*(f_i),\beta_*( f_j)) = \lambda_{M_r}(f_i,f_j)= 0, \text{\ for all\ } 1\leq i, j \leq r.
 \end{equation}

Let $U = \la \beta_*(F^*); H; F\oplus F^*\ra$ denote the submodule of $H_2(M_r)$ generated by $H(F) = F \oplus F^*$, together with the classes $\{\beta_*(f^*_i)\}$, and the classes $\{e_i\}$, for $1\leq i \leq r$. Then we claim that $U \cong \beta_*(F^*) \oplus H  \oplus H(F)$ is a \emph{free direct summand} of $H_2(M_r)$, with indicated basis elements,   on which the restriction of $\lambda_{M_r}$ is \emph{a non-singular  form}. 
 
\medskip   

We  check that  $U \cong \beta_*(F^*) \oplus H \oplus F\oplus F^*$ is a free submodule (of rank $4r$) in $H_2(M_r)$
 by first showing that 
 $$\beta_*(F^*)  \cap (H \oplus F \oplus F^*) = 0.$$
 We then observe that the restriction  $\lambda_U$ of the intersection form to  $U \subset H_2(M_r)$ is non-singular. It follows that $(U, \lambda_U)$ is an orthogonal direct summand of $(H_2(M_r), \lambda_{M_r})$.

\medskip
Here are the details: suppose that $u \in \beta_*(F^*) \cap (H \oplus F \oplus F^*)$. 
We can express 
 $$u = \sum a_i \beta_*(f_i^*)  = \sum b_i e_i +\sum c_i f_i + \sum d_i f_i^*$$
 as a $\La$-linear combinations of the basis elements. Then by the formula \eqref{eq:seventwo} above, we have  $ \lambda_{M_r}(\beta_*(f^*_i), e_j)  = \delta_{ij}$.  and hence
 $\lambda_{M_r}(u, e_i)  = a_i$. 
 Since $H$ is totally isotropic by \eqref{eq:sevenfour},
  the summand $H \oplus F \oplus F^*$ is orthogonal to $H$, and it follows that $\lambda_{M_r}(u, e_i) = 0$. Hence all the $a_i$ are zero and $u=0$.

 Now let $\lambda_U$ denote the restriction of $\lambda_{M_r}$ to $U$. The submodule $H \oplus F$ is a totally isotropic, based  free direct summand of rank $2r$ in $U$, and the dual basis elements under $\lambda_U$ form the basis of the complementary direct summand $\beta_*(F^*) \oplus F^*$. Hence $\lambda_U$ is non-singular, and in fact
 $\lambda_U \cong \Met(H \oplus F, g)$, where $g$ encodes the intersections of $\beta_*(F^*)$ with  $F^*$ (which may be non-zero). In this situation, it follows that $\lambda_U \cong H(\La^{2r})$ is isomorphic to a non-singular hyperbolic form (see \cite[Lemma 5.3]{wall-book}).

 \medskip
 Hence there is a splitting for the intersection form 
$$\lambda_{M_r} = \Met(H_2(K), g) \perp H(F)
\cong (E, \lambda_0)\perp \lambda_U$$
with respect to the orthogonal complement $(E, \lambda_0)=(\lambda_U)^\perp$. Since $M$ has good fundamental group and $\lambda_{M_r}$ contains the hyperbolic subform
$$ \lambda_U 
\cong H(\La^{2r}) \cong H(\La^r) \perp H(F),$$
 topological surgery \cite[Corollary 1.4]{Powell:2020} shows that $M \approx M_0 \Sharp 2r(S^2 \times S^2)$. The resulting closed topological $4$-manifold $M_0$ has $\chi(M_0) = 2 \chi(X)$. 
 \end{proof}
 
The construction of the manifold $M(X):= M_0$ completes the proof of Theorem  \ref{prop:fakeembed}.  \qed


\providecommand{\bysame}{\leavevmode\hbox to3em{\hrulefill}\thinspace}
\providecommand{\MR}{\relax\ifhmode\unskip\space\fi MR }
\providecommand{\MRhref}[2]{%
  \href{http://www.ams.org/mathscinet-getitem?mr=#1}{#2}
}
\providecommand{\href}[2]{#2}

\end{document}